\documentclass[10pt, a4paper]{amsart}

\usepackage{etex}
\usepackage[applemac]{inputenc}
\usepackage{textcomp} 
\usepackage{graphicx} 
\usepackage{flafter}
\usepackage{verbatim}
\usepackage[all]{xy}
\usepackage{epsfig,enumerate}
\usepackage{algorithmic}
\usepackage{amsmath}
\usepackage{amsthm}
\usepackage{amssymb}
\usepackage{mathrsfs}
\usepackage{amsfonts}
\usepackage{amscd}
\usepackage{euscript}
\usepackage{amssymb}
\usepackage{mathrsfs}
\usepackage{marvosym}
\usepackage[nice]{nicefrac}
\usepackage{bm}
\usepackage{parallel}
\usepackage[pdftex,bookmarks,breaklinks]{hyperref}
\numberwithin{equation}{section}
\usepackage{memhfixc}
\usepackage{pdfsync}
\usepackage{microtype}
\usepackage{booktabs}
\usepackage[a4paper,top=3cm,bottom=3cm,left=3.50cm,right=3.50cm, bindingoffset=5mm]{geometry}
\usepackage{setspace}
\usepackage{fancyhdr}
\usepackage{indentfirst}
\usepackage{footnote}
\usepackage{booktabs}
\usepackage{listings}
\usepackage{enumitem}
\usepackage{dsfont}
\usepackage{xcolor}
\usepackage{scrextend}
\usepackage{chngcntr}
\newcommand{\set}[1]{\left\{ #1 \right\}}

\newcommand{\ra}{\rightarrow}

\newcommand{\isom}{\xrightarrow{\sim}}
\newcommand{\map}[5]{#1: \xymatrix@R=0in{#2 \ar[r] & #3 \\ #4 \ar@{|->}[r] & #5}}

\newcommand{\chr}{\mathrm{char}}

\newcommand{\alg}[1]{\overline{#1}}
\newcommand{\sep}[1]{{#1}^{\rm{s}}}
\newcommand{\Gal}{\mathrm{Gal}}
\newcommand{\Frob}{\mathrm{Frob}}

\newcommand{\Ind}{\mathrm{Ind}}

\newcommand{\tr}{\mathrm{tr}\,}

\newcommand{\Z}{\mathbb{Z}}
\newcommand{\Q}{\mathbb{Q}}

\newcommand{\C}{\mathbb{C}}

\newcommand{\GL}{\mathrm{GL}}

\newcommand{\gl}{\mathfrak{gl}}

\theoremstyle{plain}
\newtheorem{theorem}{Theorem}
\newtheorem{proposition}[theorem]{Proposition}
\newtheorem{corollary}[theorem]{Corollary}
\newtheorem{lemma}[theorem]{Lemma}
\newtheorem{conjecture}[theorem]{Conjecture}
\theoremstyle{definition}

\newtheorem{definition}[theorem]{Definition}

\newtheorem{remark}[theorem]{Remark}

\numberwithin{equation}{section}
\numberwithin{theorem}{section}

\begin{document}

\title{Independence of Algebraic Monodromy Groups in Compatible Systems}
\author{Federico Amadio Guidi}
\date{}
\address{Mathematical Institute \\ University of Oxford \\ Oxford, UK}
\email{federico.amadio@maths.ox.ac.uk}

\begin{abstract}
In this paper we develop a general method to prove independence of algebraic monodromy groups in compatible systems of representations, and we apply it to deduce independence results for compatible systems both in automorphic and in positive characteristic settings. In the abstract case, we prove an independence result for compatible systems of Lie-irreducible representations, from which we deduce an independence result for compatible systems admitting what we call a Lie-irreducible decomposition. In the case of geometric compatible systems of Galois representations arising from certain classes of automorphic forms, we prove the existence of a Lie-irreducible decomposition, assuming a classical irreducibility conjecture. From this we deduce an independence result. We conclude with the case of compatible systems of representations of the absolute Galois group of a global function field, for which we prove the existence of a Lie-irreducible decomposition, and we deduce an independence result. From this we also deduce an independence result for compatible systems of lisse sheaves on normal varieties over finite fields.
\end{abstract}

\maketitle

\section*{Introduction}

Questions on $\ell$-independence of algebraic monodromy groups in compatible systems originate in the literature from a classical result by Serre, see \cite{serre68}, which states that if $E$ is an elliptic curve without complex multiplication over a number field, then the image of the Galois representation on the $\ell$-adic Tate module of $E$ is an open subgroup of $\GL_2 (\Z_\ell)$ for every prime $\ell$, and is equal to $\GL_2 (\Z_\ell)$ for all but finitely many primes $\ell$. The problem of extending this result to more general geometric settings led, for instance, to the formulation of the Mumford-Tate conjecture for abelian varieties, see \cite{mum66}. Nevertheless, one can naturally study $\ell$-independence questions for any abstract rational compatible system of representations. In this context, the foundational work of Larsen and Pink, see \cite{LP92}, provides some breakthroughs.

It is immediate to notice how the main results of \cite{LP92} can be extended verbatim to compatible systems having coefficients in any number field. In our paper, we use these results to develop a general method to prove $\lambda$-independence of the neutral components of the algebraic monodromy groups of compatible systems over a finite extension of the field of coefficients, and at a set of places of residual Dirichlet density $1$. The main idea is that one can prove a $\lambda$-independence result of this form whenever a compatible system admits a \emph{Lie-irreducible decomposition} over a finite extension of its field of coefficients. Let us briefly sketch our approach. First of all, we apply the results of \cite{LP92} to prove $\lambda$-independence of the neutral components of the algebraic monodromy groups of compatible systems of Lie-irreducible representations over a finite extension of the field of coefficients, and at a set of places of residual Dirichlet density $1$, see Proposition \ref{independence_theorem}. We then introduce the notion of Lie-irreducible decomposition, see Definition \ref{Lie-irreducible_decomposition_definition}, and we deduce from the Lie-irreducible case an analogous $\lambda$-independence result for compatible systems admitting a Lie-irreducible decomposition, see Corollary \ref{independence_Lie-irreducible_decomposition}.

In the case of automorphic compatible systems (and actually for a slightly larger class of geometric compatible systems), assuming a classical conjecture on the irreducibility of automorphic Galois representations for a set of primes of Dirichlet density $1$, we prove the existence of a Lie-irreducible decomposition, see Theorem \ref{decomposition_theorem_automorphic}, from which we deduce a $\lambda$-independence result, see Corollary \ref{independence_automorphic}.

In the positive characteristic setting, we prove the existence of a Lie-irreducible decomposition for compatible systems of representations of the absolute Galois group of a global function field, see Theorem \ref{decomposition_theorem_curves}, from which we deduce a $\lambda$-independence result, see Corollary \ref{independence_positive_characteristic}. We also deduce a $\lambda$-independence result for compatible systems of lisse sheaves on normal varieties over finite fields, see Corollary \ref{independence_lisse_sheaves}.

The structure of this paper is the following. In \S \ref{algebraic_monodromy_groups} we introduce the basic terminology, while in \S \ref{formal_character_variety_characteristic_polynomials} and \S \ref{frobenius_tori} we recall the main tools from \cite{LP92} which will be useful in \S \ref{independence_absolutely_irreducible}. The general method of this paper and the main results in the general setting are presented in \S \ref{independence_absolutely_irreducible}. Finally, we see applications of the results of \S \ref{independence_absolutely_irreducible} to the case of geometric compatible systems of Galois representations in \S \ref{independence_geometric_compatible_systems}, and to compatible systems in the positive characteristic case in \S \ref{independence_compatible_systems_lisse_sheaves}.

\subsection*{Acknowledgements}

The author would like to thank his PhD supervisor Andrew Wiles for introducing him to the study of these problems, and for his constant guidance and encouragement. The author would also like to thank Wojciech Gajda for bringing \cite{BGP19} to his attention, and Laura Capuano, Toby Gee, Minhyong Kim, Giacomo Micheli, and Damian R\"ossler for interesting conversations and useful suggestions.

\subsection*{Notation}

For a prime $\ell$, we denote by $\zeta_\ell$ a primitive $\ell^{\mathrm{th}}$-root of $1$.

Given a profinite group $\Gamma$, an integer $n \geq 1$, and a continuous representation $\rho : \Gamma \ra \GL_n (\alg{\Q}_\ell)$, we let $\overline{\rho} : \Gamma \ra \GL_n (\alg{\mathbb{F}}_\ell)$ denote the semisimplification of the reduction of $\rho$ modulo $\ell$, which is defined up to conjugacy.

Given a field $k$, we let $\alg{k}$ be an algebraic closure of $k$, and $\sep{k}$ be a separable closure of $k$ inside $\alg{k}$, and we denote by $\Gamma_k = \Gal (\sep{k} / k)$ the absolute Galois group of $k$.

Given a global field $F$, we denote by $|F|$ the set of finite places of $F$. If $\Sigma$ is a set of places of $F$, and $p$ is a prime, we set $\Sigma_{\neq p} = \set{v \in \Sigma \, : \, v \nmid p}$. We denote by $\mathbb{A}_F$ the ring of adeles of $F$.

If $K$ is a non-Archimedean local field, we let $I_K$ be the inertia subgroup of $\Gamma_K$, and $\Frob_K \in \Gamma_K / I_K$ be the geometric Frobenius. When $K = F_v$, for $F$ a global field, and $v$ a finite place of $F$, we write $\Frob_v = \Frob_{F_v}$.

If $K$ is a finite extension of $\Q_p$, given a Weil-Deligne representation $(r, N)$ of the Weil group $W_K$ of $K$, we let $(r, N)^{\mathrm{F-ss}}$ denote its Frobenius-semisimplification. Given a continuous representation $\rho : \Gamma_K \ra \GL_n (\alg{\Q}_\ell)$, with $\ell \neq p$, we denote by $\mathrm{WD} (\rho)$ the corresponding Weil-Deligne representation of $W_K$. We denote by $\mathrm{rec}_K$ the local Langlands correspondence for $\GL_n$ over $K$ of \cite{HT01}. When $K = F_v$, for $F$ a number field, and $v$ a finite place of $F$, we write $\mathrm{rec}_v = \mathrm{rec}_{F_v}$.

If $K$ is a finite extension of $\Q_\ell$, if $V$ is a finite dimensional vector space over $\alg{\Q}_\ell$, and $\rho : \Gamma_K \ra \GL (V)$ is a continuous semisimple de Rham representation, we denote by $\mathrm{WD} (\rho)$ the corresponding Weil-Deligne representation of $W_K$. Also, given an embedding $\tau : K \ra \alg{\Q}_\ell$ we define the multiset $\mathrm{HT}_\tau (\rho)$ of \emph{$\tau$-Hodge-Tate weights} of $\rho$ to be the multiset of $\dim_{\alg{\Q}_\ell} V$ integers containing $i$ with multiplicity $\dim_{\alg{\Q}_\ell} (V \otimes_{\tau, K} \widehat{\alg{K}} (i))^{\Gamma_F}$, where $\widehat{\alg{K}}$ denotes the completion of an algebraic closure of $K$. If $F$ is a number field, if $V$ is a finite dimensional vector space over $\alg{\Q}_\ell$, and $\rho : \Gamma_F \ra \GL (V)$ is a continuous semisimple representation which is de Rham at all places above $\ell$, given an embedding $\tau : F \ra \alg{\Q}_\ell$ we define the multiset of \emph{$\tau$-Hodge-Tate weights} of $\rho$ to be $\mathrm{HT}_\tau (\rho) = \mathrm{HT}_\tau (\rho \! \mid_{\Gamma_{F_{v(\tau)}}})$, where $v(\tau)$ is the place of $F$ induced by $\tau$.

Given a scheme $X$ of finite type over $\mathbb{F}_q$, where $q$ is a power of $p$, we denote by $|X|$ its set of closed points. For each $x \in |X|$, let $\overline{x}$ be an algebraic geometric point of $X$ localised at $x$, and let $\Frob_x \in \Gamma_{k(x)} = \Gal (k(\overline{x}) / k(x))$ be the geometric Frobenius. Given a sheaf $\mathcal{F}$ on $X$, we write $\det (1 - \Frob_x t, \mathcal{F})$ for the characteristic polynomial of the image of $\Frob_x$ under the monodromy representation of $\Gamma_{k(x)}$ on $\mathcal{F}_{k(\overline{x})}$.

\section{Compatible systems and algebraic monodromy groups} \label{algebraic_monodromy_groups}

A \emph{group with Frobenii}, or simply an \emph{$F$-group}, is a pair $(\Gamma, \set{F_\alpha}_{\alpha \in A})$ where $\Gamma$ is a profinite group, and $\set{F_\alpha}_{\alpha \in A}$ is a collection of elements of $\Gamma$, called \emph{Frobenius elements}, which are dense in $\Gamma$. Motivating examples for the definition of group with Frobenii are clearly the absolute Galois group of a global field, together with the collection of geometric Frobenii at its finite places, and the \'etale fundamental group of a scheme of finite type over a finite field, together with the collection of geometric Frobenii at its closed points.

Let $(\Gamma, \set{F_\alpha}_{\alpha \in A})$ be a group with Frobenii, let $E$ be a number field, and let $\Lambda$ be a set of finite places of $E$ of residual Dirichlet density\footnote{We say that a set $\Sigma$ of finite places of a number field $F$ has \emph{residual Dirichlet density} $\sigma$ if the set of rational primes $\set{\chr (k(v)) \, : \, v \in \Sigma}$, where $k(v)$ denotes the residue field of $F$ at $v$, has Dirichlet density $\sigma$ in $\Z$.} $1$. We give the following definition\footnote{Compare \cite[Definition 6.5]{LP92}.}.

\begin{definition} \label{compatible_systems_definition}
A \emph{compatible system} of rank $n$ representations of $\Gamma$ defined over $E$ is a family $\mathcal{R} = \set{\rho_{\lambda}}_{\lambda \in \Lambda}$ of continuous representations
\[ \rho_{\lambda} : \Gamma \ra \GL_n (\alg{E}_{\lambda}) \]
such that there is a subset $\mathcal{X} \subset A \times \Lambda$~satisfying the following conditions.
\begin{enumerate}
\item[$(1)$] For every $\alpha \in A$, $(\alpha, \lambda) \in \mathcal{X}$ for all but finitely many $\lambda \in \Lambda$.
\item[$(2)$] For all $(\alpha, \lambda) \in \mathcal{X}$, the characteristic polynomial of $\rho_{\lambda} (F_\alpha)$ has coefficients in~$E$~and is independent of $\lambda$.
\item[$(3)$] For any places $\lambda_1, \ldots, \lambda_m \in \Lambda$, the set $\set{F_\alpha \, : \, (\alpha, \lambda_i) \in \mathcal{X} \text{ for all } i = 1, \ldots, m}$ is dense in~$\Gamma$.
\end{enumerate}
\end{definition}

In the last two sections we will see two main (motivating) classes of examples of compatible systems, i.e. geometric compatible systems of Galois representations, see \S \ref{independence_geometric_compatible_systems}, and compatible systems of lisse sheaves on schemes of finite type over finite fields, see \S \ref{independence_compatible_systems_lisse_sheaves}.

We say that a compatible system $\mathcal{R} = \set{\rho_{\lambda}}_{\lambda \in \Lambda}$ of rank $n$ representations of $\Gamma$ defined over $E$ has \emph{coefficients} in a finite extension $E'$ of $E$ if for any $\lambda \in \Lambda$ the image of $\rho_{\lambda}$ is contained in $\GL_n (E'_{\lambda'})$ for some place $\lambda'$ of $E'$ above $\lambda$.

Two compatible systems $\mathcal{R} = \set{\rho_{\lambda}}_{\lambda \in \Lambda}$ and $\mathcal{R}' = \set{\rho_{\lambda}'}_{\lambda \in \Lambda'}$ defined over a number field $E$ are said to be \emph{isomorphic} if there exists a subset $\Lambda'' \subset \Lambda \cap \Lambda'$ of residual Dirichlet density $1$ such that $\rho_\lambda \cong \rho_\lambda'$ for all $\lambda \in \Lambda''$. In this case, we write $\mathcal{R} \cong \mathcal{R}'$. When $\mathcal{R}$ and $\mathcal{R}'$ have coefficients in a number field $E'$, we say that they are \emph{isomorphic over $E$} if the isomorphism $\rho_\lambda \cong \rho_\lambda'$ is over $E$ for all $\lambda \in \Lambda''$.

Let $(\Gamma, \set{F_\alpha}_{\alpha \in A})$ be a group with Frobenii, and let $\mathcal{R} = \set{\rho_{\lambda}}_{\lambda \in \Lambda}$ be a compatible system of rank $n$ representations of $\Gamma$ with coefficients in a number field $E$. For any $\lambda \in \Lambda$, denote by~$G_\lambda$~the Zariski closure of $\rho_\lambda (\Gamma)$ in $\GL_n (E_\lambda)$. This is an algebraic subgroup of $\GL_{n, E_\lambda}$, which we call the \emph{algebraic monodromy group} of $\rho_\lambda$. Denote by $G_\lambda^\circ$ the connected component of the identity of $G_\lambda$. We recall the following lemma.

\begin{lemma} [{\cite[Lemma 6.9]{LP92}}]
If $\rho_\lambda$ is semisimple, then $G_\lambda$ is reductive. If $\rho_\lambda$ is absolutely irreducible, then $G_\lambda$ is reductive, and its natural representation $G_\lambda \ra \GL_{n, E_{\lambda}}$~is absolutely irreducible.
\end{lemma}

In the next sections, we consider the problem of studying the algebraic monodromy groups $G_\lambda$ of $\mathcal{R} = \set{\rho_\lambda}_{\lambda \in \Lambda}$ varying $\lambda$. The hope in this context is to find a global algebraic subgroup $G$ of $\GL_{n, E}$ defined over $E$ such that $G_\lambda$ is conjugate to $G \times_E E_\lambda$ over $E_\lambda$ for any $\lambda \in \Lambda$. This is the prototype of a \emph{$\lambda$-independence} problem for the compatible system $\mathcal{R}$. Unfortunately, proving a $\lambda$-independence result of this form in the abstract setting is actually too optimistic. For this reason, we will focus on weaker $\lambda$-independence statements.

\section{The formal character and the variety of characteristic polynomials} \label{formal_character_variety_characteristic_polynomials}
 
If $G$ is a reductive group over a field $k$ of characteristic zero, $T \subset G$ is a maximal torus of $G$ over~$k$, and $\rho : G \ra \GL_{n, k}$ is a representation of $G$ over $k$, we define the \emph{formal character} of $\rho$ to be the restriction of $\rho$ to $T$.

Let $\mathrm{ch} : \GL_{n, k} \ra \mathbb{G}_{m, k} \times \mathbb{A}^{n-1}_k$ be the morphism which associates to a matrix the coefficients of its characteristic polynomial. Given a connected reductive algebraic subgroup~$G$~of~$\GL_{n, k}$, and a maximal torus $T \subset G$ over $k$, we have that $\mathrm{ch} (G) = \mathrm{ch} (T)$ is a variety defined over $\Q$, which determines uniquely the formal character of the natural representation $G \ra \GL_{n, k}$ up to isomorphism, see \cite[\S 4]{LP92}.

Let $\mathbb{G}_{m, k}^n$ be the split maximal torus of $\GL_{n, k}$. The Weyl group of $\GL_{n, k}$ with respect to $\mathbb{G}_{m, k}^n$ is the symmetric group $\mathfrak{S}_n$, which acts by permutation of factors. Let $T_0 \subset \mathbb{G}_{m, k}^n$ be a subtorus such that $T_0 \times_k \alg{k}$ is conjugate to $T \times_k \alg{k}$ over $\alg{k}$. Since the semisimple part of any point in $G$ can be conjugate into $T_0$ over $\alg{k}$, we have that~$\mathrm{ch} (G) = \mathrm{ch} (T_0)$ pointwise. For every~$\sigma \in \mathfrak{S}_n \backslash \mathrm{Cent}_{\mathfrak{S}_n} (T_0)$, we define a proper subgroup $H_\sigma \subset T_0$ in the following way. We let $H_\sigma = \set{t \in T \, : \, \sigma(t) = t}$ if $\sigma(T_0) = T_0$, and $H_\sigma = T_0 \cap \sigma (T_0)$ otherwise. We define now $Y$ to be the union of all $\mathrm{ch} (H_\sigma)$. This is a Zariski-closed proper subset of~$\mathrm{ch} (T_0) = \mathrm{ch} (G)$. We say that a point $g \in G$ is \emph{$\Gamma$-regular} if $\mathrm{ch} (g) \notin Y$. We have the following result.

\begin{proposition} [{\cite[Proposition 4.7]{LP92}}] \label{gamma_regular_tori}
$\phantom{aaaa}\\$
\vspace{-0.5cm}
\begin{enumerate}
\item[$(1)$] For every $x \in (\mathrm{ch} (G) \setminus Y) (k)$, there exist a torus $T \subset \GL_{n, k}$, and an element $t \in T(k)$, such that $\mathrm{ch} (T) = \mathrm{ch} (G)$ and $\mathrm{ch} (t) = x$. The pair $(t, T)$ is unique up to conjugation by $\GL_n (k)$.
\item[$(2)$] Let $g \in G (k)$ be $\Gamma$-regular. Then $g$ lies in a unique maximal torus $T$ of $G$, and the $\GL_n (k)$-conjugacy class of $(g, T)$ in uniquely determined by $\mathrm{ch} (g)$ and $\mathrm{ch} (G)$.
\end{enumerate}
\end{proposition}

Let $(\Gamma, \set{F_\alpha}_{\alpha \in A})$ be a group with Frobenii, and let $\mathcal{R} = \set{\rho_{\lambda}}_{\lambda \in \Lambda}$ be a compatible system of rank $n$ semisimple representations of $\Gamma$ with coefficients in a number field $E$. Keep notations as in \S \ref{algebraic_monodromy_groups}. We have:

\begin{proposition} [{\cite[Proposition 6.12]{LP92}}] \label{independence_formal_character}
The formal character of the natural representation $G_{\lambda}^\circ \times_{E_\lambda} \alg{E}_\lambda \ra \GL_{n, \alg{E}_\lambda}$ is independent of $\lambda$. More precisely, there exist a torus~$T_0$~over $E$ and a faithful representation $\rho_0 : T_0 \ra \GL_{n, E}$ such that $T_0 \times_{E} \alg{E}_\lambda$ is isomorphic to a maximal torus in $G_{\lambda}^\circ \times_{E_\lambda} \alg{E}_\lambda$, and $\rho_0 \otimes_{E} \alg{E}_\lambda$ is equivalent to the formal character of~$G_{\lambda}^\circ \times_{E_\lambda} \alg{E}_\lambda \ra \GL_{n, \alg{E}_\lambda}$ for all $\lambda \in \Lambda$.
\end{proposition}

\begin{proposition} [{\cite[Proposition 6.14]{LP92}}] \label{independence_connected_components}
There is an open normal subgroup~$\Gamma' \subset \Gamma$~such that $\rho_\lambda$ induces an isomorphism $\Gamma / \Gamma' \isom G_\lambda / G_\lambda^\circ$ for all $\lambda \in \Lambda$.
\end{proposition}

Note that Proposition \ref{independence_connected_components} implies that the group $G_\lambda / G_\lambda^\circ$ of connected components of $G_\lambda$ is independent of $\lambda$. Also, if we let $A' = \set{\alpha \in A \, : \, F_\alpha \in \Gamma'}$, then the pair $(\Gamma', \set{F_\alpha}_{\alpha \in A'})$ is a group with Frobenii, and $\mathcal{R}' = \set{\rho'_\lambda = \rho_\lambda \! \mid_{\Gamma'}}_{\lambda \in \Lambda}$ forms a compatible system, with $\mathcal{X}' = \mathcal{X} \cap (A' \times \Lambda)$. Clearly, the algebraic monodromy group of~$\rho'_\lambda$~is~$G_\lambda^\circ$~for every $\lambda \in \Lambda$. 

\section{Frobenius tori and the splitting field of a compatible system} \label{frobenius_tori}

Let $(\Gamma, \set{F_\alpha}_{\alpha \in A})$ be a group with Frobenii, and let $\mathcal{R} = \set{\rho_{\lambda}}_{\lambda \in \Lambda}$ be a compatible system of rank $n$ semisimple representations of $\Gamma$ with coefficients in a number field $E$. Let us assume in this section that for every $\lambda \in \Lambda$ the reductive group $G_\lambda$ is connected. We have the following result.

\begin{proposition} [{\cite[Proposition 7.2]{LP92}}] \label{gamma_regular}
For any $\lambda \in \Lambda$, the set of all $\gamma \in \Gamma$ such that $\rho_\lambda (\gamma)$ is $\Gamma$-regular is open and dense in $\Gamma$.
\end{proposition}

For every $(\alpha, \lambda) \in \mathcal{X}$, the point $\mathrm{ch} (\rho_\lambda (F_\alpha)) \in (\mathbb{G}_{m, E_\lambda} \times \mathbb{A}^{n-1}_{E_\lambda}) (E)$ depends only on $\alpha$, and so the condition on $\rho_\lambda (F_\alpha)$ being $\Gamma$-regular does depend only on $\alpha$ as well. When this condition holds, we simply say that $\alpha$ is $\Gamma$-regular. By Proposition \ref{gamma_regular}, if we replace $A$ with the set of $\Gamma$-regular $\alpha \in A$ the corresponding set $\mathcal{X} \subset A \times \Lambda$ still satisfies the conditions $(1)$-$(3)$ in Definition \ref{compatible_systems_definition}. From now on, let us assume that every $\alpha \in A$ is $\Gamma$-regular.

By Proposition \ref{gamma_regular_tori}, for every $\alpha \in A$ we are then given a torus $T_{\alpha} \subset \GL_{n, E}$, canonical up to conjugacy, associated to $\mathrm{ch} (\rho_\lambda (F_\alpha))$, for $\lambda \in \Lambda$ such that $(\alpha, \lambda) \in \mathcal{X}$, so that the set~$\mathcal{X}$~satisfies the following conditions.
\begin{enumerate}
\item[$(1)$] For every $\alpha \in A$, $(\alpha, \lambda) \in \mathcal{X}$ for all but finitely many $\lambda \in \Lambda$.
\item[$(2)$] For all $(\alpha, \lambda) \in \mathcal{X}$, $T_\alpha \times_{E} E_\lambda$ is conjugate to a maximal torus of $G_\lambda$ over $E_\lambda$.
\item[$(3)$] For all pairwise distinct places $\lambda_1, \ldots, \lambda_k \in \Lambda$, and all maximal tori $T_i$ of $G_{\lambda_i}$, there exists an $\alpha \in A$ such that, for every $i$, $(\alpha, \lambda_i) \in \mathcal{X}$ and $T_\alpha \times_{E} E_{\lambda_i}$ is conjugate to $T_i$ over $E_{\lambda_i}$.
\end{enumerate}

The tori $T_\alpha$ over $E$ are called \emph{Frobenius tori}. For every $\alpha \in A$, let $L_\alpha$ be the splitting field of $T_\alpha$ over $E_\alpha$, that is the smallest degree extension $L_\alpha / E$ such that $T_\alpha \times_{E} L_\alpha$ is split. Let $L$ be the intersection of all $L_\alpha$. We call this field the \emph{splitting field} of the compatible system $\mathcal{R}$.

The role of Frobenius tori and the splitting field of a compatible system is clarified by the following results.

\begin{proposition} [{\cite[Proposition 8.9]{LP92}}] \label{unramified}
There exists a subset $\Lambda'$ of $\Lambda$ of residual Dirichlet density $1$ such that for every $\lambda \in \Lambda'$ the connected reductive group $G_\lambda$ is unramified, and split over $L_{\mathfrak{l}}$, for some finite place $\mathfrak{l}$ of $L$ lying above $\lambda$.
\end{proposition}

\begin{corollary} \label{independence_formal_character_split}
Let $\Lambda'$ be the subset of $\Lambda$ where $G_\lambda$ is unramified. There exist a $E$-split torus~$T_0$~and a faithful representation $\rho_0 : T_0 \ra \GL_{n, E}$ such that for all $\lambda \in \Lambda'$ there exists a finite place $\mathfrak{l}$ of $L$ lying above $\lambda$ such that $T_0 \times_{E} L_{\mathfrak{l}}$ is conjugate to a maximal torus in $G_\lambda \times_{E_\lambda} L_{\mathfrak{l}}$, and $\rho_0 \otimes_{E} L_{\mathfrak{l}}$ is equivalent to the formal character of $G_\lambda \times_{E_\lambda} L_{\mathfrak{l}} \ra \GL_{n, L_{\mathfrak{l}}}$.
\end{corollary}

\begin{theorem} [{part of \cite[Theorem 9.4]{LP92}}] \label{independence_root_datum}
Assume that each $\rho_\lambda$ is absolutely irreducible, and that $L = E$. Let $\Lambda'$ be the subset of $\Lambda$ where $G_\lambda$ is unramified. Then, for all $\lambda \in \Lambda'$ the root datum~$\Psi_{\lambda}$~of~$G_\lambda$~is independent of $\lambda$ up to isomorphism.
\end{theorem}

\section{Independence in compatible systems of Lie-irreducible representations} \label{independence_absolutely_irreducible}

The following result holds.

\begin{proposition} \label{independence_theorem}
Let $(\Gamma, \set{F_\alpha}_{\alpha \in A})$ be a group with Frobenii, and let $\mathcal{R} = \set{\rho_{\lambda}}_{\lambda \in \Lambda}$ be a compatible system of rank $n$ absolutely Lie-irreducible\footnote{We say that a rank $n$ representation $\rho$ of $\Gamma$ with coefficients in a field $k$ of characteristic zero is \emph{absolutely Lie-irreducible} if $\rho \! \mid_{\Gamma'}$ is absolutely irreducible for any open subgroup $\Gamma'$ of $\Gamma$.} representations of $\Gamma$ with coefficients in a number field~$E$. Then, there exist a finite extension $L$ of $E$, a subset $\Lambda'$ of $\Lambda$ of residual Dirichlet density~$1$, and a split connected reductive algebraic subgroup $G$ of $\GL_{n, L}$ defined over $L$ such that for each $\lambda \in \Lambda'$ there exists a finite place $\mathfrak{l}$ of $L$ lying above $\lambda$ such that $G_{\lambda}^\circ \times_{E_{\lambda}} L_{\mathfrak{l}}$ is conjugate to $G \times_{L} L_{\mathfrak{l}}$ over $L_{\mathfrak{l}}$.
\end{proposition}

\begin{proof}
Let $\Gamma'$ be the open subgroup of $\Gamma$ given by Proposition \ref{independence_connected_components}. As each~$\rho_\lambda$~is absolutely Lie-irreducible, the restrictions $\rho_\lambda \! \mid_{\Gamma'}$ are still absolutely irreducible. Up to restrict to $\Gamma'$, we can assume each $G_\lambda$ to be connected. Let~$L$~be the splitting field of $\mathcal{R}$, and let $\Lambda' \subset \Lambda$ be as in Proposition \ref{unramified}. Let $\mathfrak{L}'$ be the set of finite places of $L$ such that for every $\lambda \in \Lambda'$ there exists $\mathfrak{l} \in \mathfrak{L}'$ lying above $\lambda$ such that~$G_\lambda$~is split over $L_{\mathfrak{l}}$. This set has residual Dirichlet density $1$. Let $\rho_{\mathfrak{l}} = \rho_\lambda \otimes_{E_\lambda} L_{\mathfrak{l}}$, and consider the compatible system $\set{\rho_{\mathfrak{l}}}_{\mathfrak{l} \in \mathfrak{L}'}$. Clearly, the algebraic monodromy group $G_{\mathfrak{l}}$ of $\rho_{\mathfrak{l}}$ is equal to $G_\lambda \times_{E_\lambda} L_{\mathfrak{l}}$ by construction. By Theorem \ref{independence_root_datum}, the root datum $\Psi_{\mathfrak{l}}$ of $G_{\mathfrak{l}}$ is independent of~$\mathfrak{l}$~up to isomorphism. Fix such a root datum $\Psi$. By \cite[Theorem A.4.6]{CGP15}, there exists a (unique up to isomorphism) split connected reductive group $G$ over $L$ with root datum $\Psi$, and such that $G \times_L L_{\mathfrak{l}}$ is conjugate to $G_{\mathfrak{l}}$ over $L_{\mathfrak{l}}$ for any $\mathfrak{l} \in \mathfrak{L}$.
\end{proof}

Let $(\Gamma, \set{F_\alpha}_{\alpha \in A})$ be a group with Frobenii, and let $\mathcal{R} = \set{\rho_{\lambda}}_{\lambda \in \Lambda}$ be a compatible system of rank $n$ representations of $\Gamma$ with coefficients in a number field~$E$. We introduce the following definition.

\begin{definition} \label{Lie-irreducible_decomposition_definition}
We say that $\mathcal{R}$ has a \emph{Lie-irreducible decomposition} over $E$ if there exist sets $\Lambda_i$ of places of $E$ of residual Dirichlet density $1$, open subgroups $\Gamma_i$ of $\Gamma$, compatible systems $\mathcal{S}_i = \set{ \sigma_{i, \lambda}}_{\lambda \in \Lambda_i}$ of rank $m_i$ absolutely Lie-irreducible representations of $\Gamma_i$ with coefficients in $E$, and rank $d_i$ Artin representations\footnote{In this setting, we say that a continuous representation $\omega : \Gamma \ra \GL_{d} (\alg{E}_\lambda)$ is \emph{Artin} if it factors through a finite quotient of $\Gamma$. Any Artin representation $\omega$ can be actually realised over $\alg{\Q}$, so that $\lambda$-independence questions become trivial in this context.} $\omega_i$ of $\Gamma_i$, for $i = 1, \ldots, k$, where $\sum_{i = 1}^k m_i d_i [\Gamma : \Gamma_i] = n$, such that
\[ \rho_\lambda \cong \oplus_{i = 1}^k \Ind_{\Gamma_i}^{\Gamma} (\sigma_{i, \lambda} \otimes \omega_i)\]
over $E$ for all $\lambda \in \Lambda \cap ( \cap_{i=1}^k \Lambda_i)$. In this case, we write
\[ \mathcal{R} \cong \oplus_{i = 1}^k \Ind_{\Gamma_i}^{\Gamma} (\mathcal{S}_i \otimes \omega_i). \]
\end{definition}

\begin{remark} \label{Lie-irreducible_decomposition_remark}
For some choices of the group $\Gamma$, one can prove that any $\lambda$-adic representation $\rho : \Gamma \ra \GL_n (\alg{E}_\lambda)$ has a Lie-irreducible decomposition $\rho \cong \oplus_{i = 1}^k \mathrm{Ind}_{\Gamma_i}^\Gamma (\sigma_i \otimes \omega_i)$, where $\Gamma_i$ is an open subgroup of $\Gamma$, $\sigma_i : \Gamma_i \ra \GL_{m_i} (\alg{E}_\lambda)$ is Lie-irreducible, and $\omega_i : \Gamma_i \ra \GL_{d_i} (\alg{E}_\lambda)$ is Artin, for $i = 1, \ldots, k$. This has been proved by Katz, see \cite[Proposition 1]{katz87}, when $\Gamma$ is the \'etale fundamental group of a smooth connected affine curve over an algebraically closed field of positive characteristic, and by Patrikis, see \cite[Proposition 3.4.1]{pat12}, by adapting Katz's argument, when $\Gamma$ is the absolute Galois group of a number field. To prove a decomposition result for a compatible system $\mathcal{R} = \set{\rho_{\lambda}}_{\lambda \in \Lambda}$ of representations of $\Gamma$, one would then start from a single $\lambda$-adic representation $\rho_\lambda$ in $\mathcal{R}$, decompose it as $\rho_\lambda \cong \oplus_{i = 1}^k \Ind_{\Gamma_i}^{\Gamma} (\sigma_{i, \lambda} \otimes \omega_i)$, and extend each $\sigma_{i, \lambda}$ to a compatible system $\mathcal{S}_i$ of Lie-irreducible representations. The isomorphism $\mathcal{R} \cong \oplus_{i = 1}^k \Ind_{\Gamma_i}^{\Gamma} (\mathcal{S}_i \otimes \omega_i)$ would then follow by considerations on the traces of Frobenius elements, see the proofs of Theorem \ref{decomposition_theorem_automorphic} and Theorem \ref{decomposition_theorem_curves}. However, extending a $\lambda$-adic representation to a compatible system is a highly non trivial problem, and can be achieved, for instance, under suitable assumptions in the number field case by means of potential automorphy techniques, see \S \ref{independence_geometric_compatible_systems}, and in the positive characteristic case by means of the global Langlands correspondence for $\GL_n$ of L. Lafforgue, see \S \ref{independence_compatible_systems_lisse_sheaves}.
\end{remark}

Proposition \ref{independence_theorem} has the following consequence.

\begin{corollary} \label{independence_Lie-irreducible_decomposition}
 Let $(\Gamma, \set{F_\alpha}_{\alpha \in A})$ be a group with Frobenii, and let $\mathcal{R} = \set{\rho_{\lambda}}_{\lambda \in \Lambda}$ be a compatible system of rank $n$ representations of $\Gamma$ with coefficients in a number field~$E$. Assume that $\mathcal{R}$ has a Lie-irreducible decomposition over a finite extension $E'$ of $E$. Then, there exist a finite extension $L$ of $E'$, a subset $\Lambda'$ of $\Lambda$ of residual Dirichlet density~$1$, and a connected reductive algebraic subgroup $G$ of $\GL_{n, L}$ defined over $L$ such that for each $\lambda \in \Lambda'$ there exists a finite place $\mathfrak{l}$ of $L$ lying above $\lambda$ such that $G_{\lambda}^\circ \times_{E_{\lambda}} L_{\mathfrak{l}}$ is conjugate to $G \times_{L} L_{\mathfrak{l}}$ over $L_{\mathfrak{l}}$.
\end{corollary}

\begin{proof}
By assumption, there exist a finite extension $E'$ of $E$, sets $\Lambda_i$ of places of $E'$ of residual Dirichlet density $1$, open subgroups $\Gamma_i$ of $\Gamma$, compatible systems $\mathcal{S}_i = \set{ \sigma_{i, \lambda}}_{\lambda \in \Lambda_i}$ of rank $m_i$ absolutely Lie-irreducible representations of $\Gamma_i$ with coefficients in $E'$, and rank $d_i$ Artin representations $\omega_i$ of $\Gamma_i$, for $i = 1, \ldots, k$, where $\sum_{i = 1}^k m_i d_i [\Gamma : \Gamma_i] = n$, such that
\[ \rho_\lambda \cong \oplus_{i = 1}^k \Ind_{\Gamma_i}^{\Gamma} (\sigma_{i, \lambda} \otimes \omega_i) \]
over $E'$ for all $\lambda \in \Lambda \cap ( \cap_{i=1}^k \Lambda_i)$, where we still denote by $\Lambda$ a set of places of $E'$ of residual Dirichlet density $1$ lying above the places in $\Lambda$, with a slight abuse of notations. For any $\lambda \in \Lambda \cap (\cap_{i = 1}^k \Lambda_i)$, the algebraic monodromy group $G_\lambda$ of $\rho_\lambda$ only depends on the algebraic monodromy groups of $\sigma_{i, \lambda}$ and $\omega_i$, for $i = 1, \ldots, k$.

For any $i$ and any $\lambda \in \Lambda_i$, denote by $H_{\sigma_{i, \lambda}}$ the algebraic monodromy group of $\sigma_{i, \lambda}$. By Proposition \ref{independence_theorem}, for any $i$ there exist a finite extension $L_i$ of $E'$, a subset $\Lambda'_i$ of $\Lambda_i$ of residual Dirichlet density~$1$, and a connected reductive algebraic subgroup $H_i$ of $\GL_{m_i, L_i}$ defined over $L_i$ such that for each $\lambda \in \Lambda'_i$ there exists a finite place $\mathfrak{l}$ of $L_i$ lying above $\lambda$ such that $H_{\sigma_{i, \lambda}}^\circ \times_{E'_{\lambda}} L_{\mathfrak{l}}$ is conjugate to $H_i \times_{L_i} L_{i, \mathfrak{l}}$ over $L_{i, \mathfrak{l}}$.

Since for any $i$ the representation $\omega_i$ is Artin, it can be defined over $\alg{\Q}$. Its algebraic monodromy group $H_{\omega_i}$ is then a finite subgroup of $\GL_{n, \alg{\Q}}$, which can be defined over a number field $L_{\omega_i}$. Let us enlarge $L_i$ so that it contains $L_{\omega_i}$.

The conclusion follows by taking $L$ to be the composite of the extensions $L_i$, and $\Lambda'$ to be the set of places in $\cap_{i = 1}^k \Lambda_i'$ lying above the places in $\Lambda$.
\end{proof}

\section{Geometric compatible systems of Galois representations} \label{independence_geometric_compatible_systems}

We see now an application of the abstract results of \S \ref{independence_absolutely_irreducible} to the case of geometric compatible systems of Galois representations attached to automorphic representations. In this context, we refer to \cite{BLGGT14} for the background terminology.

Let $F$ be a a number field, let $\Gamma_F$ denote the absolute Galois group of $F$, and let $S$ be a finite set of places of $F$. Let~$E$~be a number field, and let $\Lambda$ be a set of finite places of $E$ of residual Dirichlet density $1$. Let $n \geq 1$ be an integer. We recall the following definition.

\begin{definition} \label{geometric_compatible_systems_definition}
A \emph{geometric\footnote{We add ``geometric'' to the terminology of \cite[\S 5.1]{BLGGT14} in order to distinguish these compatible systems from the ``abstract'' compatible systems of Definition \ref{compatible_systems_definition}.} compatible system} of rank $n$ representations of $\Gamma_F$ defined over $E$ and unramified outside $S$ is a family $\mathcal{R} = \set{\rho_{\lambda}}_{\lambda \in \Lambda}$ of continuous semisimple representations
\[ \rho_{\lambda} : \Gamma_F \ra \GL_n (\alg{E}_{\lambda}) \]
such that:
\begin{enumerate}
\item[$(1)$] If $v \notin S$ is a finite place of $F$, then for all $\lambda$ not dividing the residue characteristic of $v$, the representation $\rho_\lambda$ is unramified at $v$, and the characteristic polynomial of $\rho_\lambda (\Frob_v)$ has coefficients in $E$ and is independent of $\lambda$.
\item[$(2)$] Each representation $\rho_\lambda$ is de Rham at all places above the residue characteristic of $\lambda$, and in fact crystalline at any place $v \notin S$ above the residue characteristic of~$\lambda$.
\item[$(3)$] For each embedding $\tau : F \ra \alg{E}$ the $\tau$-Hodge-Tate weights of $\rho_\lambda$ are independent of $\lambda$.
\end{enumerate}
\end{definition}

Clearly, any geometric compatible system $\mathcal{R}$ is a compatible system in the sense of Definition \ref{compatible_systems_definition}, with $\mathcal{X} = \set{(v, \lambda) \in |F| \times \Lambda \, : \, v \notin S, \lambda \nmid \chr (k(v))}$. All the definitions introduced in the case of general compatible systems, as well as all the results proved, are therefore valid also for geometric compatible systems of Galois representations.

Let $F$ be a CM field, let $F^+$ be the maximal totally real subfield of $F$, and let $c$ be a generator of $\Gal (F / F^+)$. Thanks to the work of many people, e.g. \cite{clozel90}, \cite{kottwitz92}, \cite{HT01}, \cite{shin11}, \cite{CH13}, geometric compatible systems can be attached to regular algebraic, conjugate self-dual, cuspidal automorphic representations of $\GL_n (\mathbb{A}_F)$. We recall here the general statement of this construction\footnote{To simplify the exposition, in this paper we decided to restrict to conjugate self-dual automorphic representations, i.e. automorphic representations $\pi$ such that the contragradient $\pi^\vee$ of $\pi$ satisfies $\pi^{\vee} \cong \pi \circ c$. An analogous construction holds in the more general case of polarised automorphic representations, see \cite[\S 2.1]{BLGGT14}. Also, an analogous construction holds when $F$ is a totally real field.}, which also involves results of \cite{TY07}, \cite{caraiani12} and \cite{caraiani14}. We have:

\begin{theorem} \label{construction_automorphic_compatible_systems_theorem}
Let $\pi$ be a regular algebraic, conjugate self-dual, cuspidal automorphic representation of $\GL_n (\mathbb{A}_F)$, unramified outside a finite set $S$ of places of $F$. Then, there exist a number field $E_\pi$, a compatible system $\mathcal{R}_\pi = \set{\rho_{\pi, \lambda}}_{\lambda \in |E_\pi|}$ of rank $n$ semisimple representations
\[ \rho_{\pi, \lambda} : \Gamma_F \ra \GL_n (E_{\pi, \lambda})\]
with coefficients in $E_\pi$, and an integer $w$ such that:
\begin{enumerate}
\item[$(1)$] $\rho_{\pi, \lambda}$ is totally odd, conjugate self-dual\footnote{With the terminology of \cite[\S 2.1]{BLGGT14}, this means that the pair $(\rho_{\pi, \lambda}, \epsilon^{1-n})$, where $\epsilon$ denotes the cyclotomic character of $\Gamma_{F^+}$, is a totally odd, polarised representation of $\Gamma_F$.}.
\item[$(2)$] If $v$ is a finite place of $F$ not dividing the residue characteristic of $\lambda$, then, given an isomorphism $\imath : \alg{E}_{\pi, \lambda} \isom \C$, we have the local-global compatibility
\[ \imath \mathrm{WD} (\rho_{\pi, \lambda} \! \mid_{\Gamma_{F_v}})^{\mathrm{F-ss}} \cong \mathrm{rec}_v (\pi_v \otimes |\det|_v^{(1-n)/2}),\]
and these Weil-Deligne representations are pure of weight $w$.
\item[$(3)$] Each representation $\rho_{\pi, \lambda}$ is de Rham at all places above the residue characteristic of $\lambda$, and for each embedding $\tau : F \ra \alg{E}_\pi$ the $\tau$-Hodge-Tate weights of $\rho_{\pi, \lambda}$ have multiplicity at most one, and are given by
\[ \mathrm{HT}_\tau (\rho_{\pi, \lambda}) = \set{a_{\tau, 1} + n-1, a_{\tau, 2} + n-2, \ldots, a_{\tau, n}}, \]
where $a = (a_{\tau, i})$ is the weight of $\pi$. Moreover
\[ \mathrm{HT}_{\tau \circ c} (\rho_{\pi, \lambda}) = \set{w - h \, : \, h \in \mathrm{HT}_\tau (\rho_{\pi, \lambda})}.\]
\item[$(4)$] If $v$ is a place of $F$ dividing the residue characteristic of $\lambda$, then, given an isomorphism $\imath : \alg{E}_{\pi, \lambda} \isom \C$, we have the local-global compatibility
\[ \imath \mathrm{WD} (\rho_{\pi, \lambda} \! \mid_{\Gamma_{F_v}})^{\mathrm{F-ss}} \cong \mathrm{rec}_v (\pi_v \otimes |\det|_v^{(1-n)/2}).\]
In particular, $\rho_{\pi, \lambda}$ is semi-stable at $v$, and if $v \notin S$ then it is crystalline.
\end{enumerate}
\end{theorem}

The compatible system $\mathcal{R}_\pi$ is then a geometric compatible system of rank $n$ representations of $\Gamma_F$ with coefficients in $E_\pi$ and unramified outside $S$, in the sense of Definition \ref{geometric_compatible_systems_definition}. Moreover, it is totally odd, conjugate self-dual, i.e. it is formed of totally odd, conjugate self-dual representations, and it is strictly pure of weight $w$, and regular.

The global Langlands conjectures in this case predict all the representations $\rho_{\pi, \lambda}$ to be absolutely irreducible. For the purposes of this paper, we assume the following weaker conjecture.

\begin{conjecture} \label{conjecture_irreducibility_automorphic_compatible_systems}
Let $\pi$ be a regular algebraic, conjugate self-dual, cuspidal automorphic representation of $\GL_n (\mathbb{A}_F)$. Then, there exists a set of rational primes $\mathcal{L}$ of Dirichlet density $1$ such that for all $\ell \in \mathcal{L}$ and $\lambda \mid \ell$ the representation
\[\rho_{\pi, \lambda} : \Gamma_F \ra \GL_n (E_{\pi, \lambda})\]
is absolutely irreducible.
\end{conjecture}

The proof of this conjecture is a widely open problem in general, and only some partial progress has been obtained so far, see for instance \cite[Theorem 5.5.2]{BLGGT14} and \cite[~Theorem 1.7]{PT15}. Assuming this conjecture, we prove the following result.

\begin{theorem} \label{decomposition_theorem_automorphic}
Let $\mathcal{R} = \set{\rho_\lambda}_{ \lambda \in \Lambda}$ be a pure, regular, totally odd, conjugate self-dual geometric compatible system of rank $n$ representations of $\Gamma_F$ with coefficients in a number field $E$. Assume Conjecture \ref{conjecture_irreducibility_automorphic_compatible_systems}. Then, $\mathcal{R}$ has a Lie-irreducible decomposition over a finite extension $E'$ of $E$.
\end{theorem}

\begin{proof}
Assuming Conjecture \ref{conjecture_irreducibility_automorphic_compatible_systems}, by \cite[Theorem 2.1]{PT15} we have that there exist subsets $\Lambda_i$ of $\Lambda$, for $i = 1, \ldots, k$, of residual Dirichlet density $1$, and strictly pure, regular, totally odd, conjugate self-dual compatible systems $\mathcal{R}_i = \set{\rho_{i, \lambda}}_{\lambda \in \Lambda_i}$ of irreducible representations $\rho_{i, \lambda} : \Gamma_F \ra \GL_{n_i} (\alg{E}_\lambda)$ for $i = 1, \ldots, k$, such that
\[ \rho_\lambda \otimes_{E_\lambda} \alg{E}_\lambda \cong \oplus_{i = 1}^k \rho_{i, \lambda}, \]
for each $\lambda \in \cap_{i = 1}^k \Lambda_i$. Furthermore, after possibly removing finitely many places from each $\Lambda_i$, we can assume that for each $\lambda \in \Lambda_i$ we have $\ell \geq 2 (n+1)$, where $\ell$ denotes the residue characteristic of $\lambda$.

Fix an $i$, and consider the compatible system $\mathcal{R}_i$. By \cite[Lemma 1.2]{PT15}, we can assume that $\mathcal{R}_i$ is defined over a CM field $E_i$. With a slight abuse of notation, we still denote by $\Lambda_i$ the corresponding set of places of $E_i$. For all $\lambda \in \Lambda_i$, let~$G_{i, \lambda}$~be the algebraic monodromy group\footnote{This would be a reductive algebraic subgroup of $\GL_{n, E_{i, \lambda'}'}$, for $E_i'$ a finite extension of $E_i$ where $\mathcal{R}_i$ takes coefficients, and $\lambda'$ a place of $E_i'$ above $\lambda$.} of~$\rho_{i, \lambda}$, and let $G_{i, \lambda}^\circ$ be the connected component of the identity of $G_{i, \lambda}$. By Lemma \ref{independence_connected_components} there exists a finite Galois extension $F'_i / F$ such that the representation $\rho_{i, \lambda}$ induces an isomorphism $\Gal (F_i' / F) \isom G_{i, \lambda} / G_{i, \lambda}^\circ$ for all $\lambda \in \Lambda_i$. By applying \cite[Proposition 5.3.2]{BLGGT14}, we get that there exists a subset $\Lambda_i^0$ of $\Lambda_i$ of residual Dirichlet density~$1$~such that if $\lambda \in \Lambda_i^0$, and $\ell$ is the residue characteristic of~$\lambda$, then $\overline{\rho}_{i, \lambda} \! \mid_{\Gamma_{F(\zeta_\ell)}}$ is absolutely irreducible. Up to removing finitely many places from $\Lambda_i^0$, we may further assume that if $\lambda \in \Lambda_i^0$, and $\ell$ is the residue characteristic of $\lambda$, then
\begin{itemize}
\item $\zeta_\ell \notin F_i'$,
\item $\ell$ is unramified in $F_i'$ and lies below no element of the set of bad places for $\mathcal{R}_i$,
\item the Hodge-Tate weights of $\rho_{i, \lambda}$ lie in a range of the form $[a_i, a_i + \ell - 2]$.
\end{itemize}

Let $F_i'^{, \mathrm{cm}}$ denote the maximal CM subfield of $F_i'$. Then, assuming Conjecture \ref{conjecture_irreducibility_automorphic_compatible_systems}, by \cite[Theorem 2.1]{PT15} we have that there exist subsets $\Lambda_i'$ of $\Lambda_i^0$ of residual Dirichlet density $1$ such that each irreducible component of $\rho_{i, \lambda} \! \mid_{\Gamma_{F_i'^{, \mathrm{cm}}}}$ is totally odd, conjugate self-dual for each $\lambda \in \Lambda_i'$.

By \cite[Lemma 5.3.1(2)]{BLGGT14} each $\rho_{i, \lambda}$ is absolutely Lie-multiplicity free\footnote{We say that a rank $n$ representation $\rho$ of a profinite group $\Gamma$ with coefficients in a field $k$ of characteristic zero is \emph{absolutely Lie-multiplicity free} if for any open subgroup $\Gamma'$ of $\Gamma$ any absolutely irreducible $\Gamma'$-subrepresentation of $\rho$ has multiplicity $1$.}. Therefore, for each $\lambda \in \Lambda_i'$, by \cite[Lemma 3.4.6(1)]{pat12} there exist an intermediate field $F \subset F_i^{\lambda} \subset F_i'$, and a rank $m_i^\lambda$ absolutely Lie-irreducible representation~$\sigma_i^\lambda$~of~$\Gamma_{F_i^\lambda}$~such that we can write~$\rho_{i, \lambda} \cong \Ind_{\Gamma_{F_i^\lambda}}^{\Gamma_F} \sigma_i^\lambda$. We now prove that there exists $\lambda_0 \in \Lambda_i'$ such that~$F_i^{\lambda_0}$~is a CM field\footnote{The following argument is largely adapted on that of \cite[Lemma 3.4.13]{pat12}.}. Denote by $F_i^{\lambda, \mathrm{cm}}$ the maximal CM subfield of $F_i^{\lambda}$. For simplicity, let us enlarge $E_i$ to contain the maximal CM subfield of $F_i'$, so that $F_i^{\lambda, \mathrm{cm}} \subset E_i$ for all $\lambda \in \Lambda_i'$, and let us take its Galois closure. Again with a slight abuse of notation, let us still denote its corresponding set of places by $\Lambda_i'$. If $F_i^{\lambda} \neq F_i^{\lambda, \mathrm{cm}}$ for all $\lambda \in \Lambda_i'$, then we can find a place $\lambda_0 \in \Lambda_i'$ with residue characteristic $\ell_0$ such that $\ell_0$ splits in $E_i$ but not in~$F_i^{\lambda_0}$. Let now~$w_0$~be a non-split place of $F_i^{\lambda_0}$ of residue characteristic~$\ell_0$. Since $F_{i, w_0}^{\lambda_0}$ does not embed in $E_{i, \lambda_0} = \Q_{\ell_0}$ by assumption, we can deduce by an argument analogous to that of~\cite[Lemma 3.4.13]{pat12} that $\rho_{i, \lambda_0}$ is not regular, which is a contradiction. Therefore, we get that $F_i^{\lambda_0} = F_i^{\lambda_0, \mathrm{cm}}$. Also, $\ell_0 \geq 2 (m_i^{\lambda_0} + 1)$ and $\zeta_{\ell_0} \notin F_i^{\lambda_0}$ by the previous assumptions. For simplicity of notation, we set $F_i = F_i^{\lambda_0}$.

The representation $\sigma_i^{\lambda_0}$ is an irreducible component of $\rho_{i, \lambda_0} \! \mid_{\Gamma_{F_i}}$, and so it is totally odd, conjugate self-dual, as $F_i \subset F_i'^{, \mathrm{cm}}$, and its set of Hodge-Tate weights is a subset of the set of Hodge-Tate weights of $\rho_{i, \lambda_0}$. This gives that $\sigma_i^{\lambda_0}$ has Hodge-Tate weights lying in $[a_i, a_i + \ell_0 - 2]$. Since $\ell_0$ is unramified in~$F_i'$, and so it is in $F_i$ as well, we have that $\sigma_i^{\lambda_0}$ is crystalline at each place $w_0$ of $F_i$ lying above $\ell_0$ by \cite[Lemma 2.2.9]{pat12}. Therefore, $\sigma_i^{\lambda_0}$ is potentially diagonalizable at each place $w_0$ of $F_i$ lying above $\ell_0$.

By applying Mackey restriction formula to $\overline{\rho}_{i, \lambda_0} \! \mid_{\Gamma_{F(\zeta_{\ell_0})}} \cong \mathrm{Res}^{\Gamma_F}_{\Gamma_{F(\zeta_{\ell_0})}} \Ind_{\Gamma_{F_i}}^{\Gamma_F} \overline{\sigma}_i^{\lambda_0}$, which is absolutely irreducible by assumption, we deduce that $\overline{\sigma}_i^{\lambda_0} \! \mid_{\Gamma_{F_i (\zeta_{\ell_0})}}$ is absolutely irreducible. Therefore, we can apply \cite[Theorem 5.5.1]{BLGGT14} to get that $\sigma_i^{\lambda_0}$ is part of a strictly pure compatible system $\mathcal{S}_i = \set{\sigma_{i, \lambda}}_{\lambda \in \Lambda_i'}$ of representations of $\Gamma_{F_i}$ defined over~$E_i$.

Since~$\mathcal{S}_i$~is regular, each representation $\sigma_{i, \lambda}$~is absolutely Lie-multiplicity free, again by \cite[Lemma 5.3.1(2)]{BLGGT14}. Also, again by assuming Conjecture \ref{conjecture_irreducibility_automorphic_compatible_systems}, we have that there exists a subset $\Lambda_i''$ of $\Lambda_i'$ of residual Dirichlet density~$1$, with $\lambda_0 \in \Lambda_i''$, such that $\sigma_{i, \lambda}$ is absolutely irreducible for each $\lambda \in \Lambda_i''$. Since $\sigma_{i, \lambda_0} \cong \sigma_i^{\lambda_0}$ is absolutely Lie-irreducible, and absolute Lie-irreducibility in a compatible system of absolutely irreducible, absolutely Lie-multiplicity free representations is independent of $\lambda$ by \cite[Corollary 3.4.11]{pat12}, we get that $\sigma_{i, \lambda}$ is absolutely Lie-irreducible for each $\lambda \in \Lambda_i''$.

For every finite place $v$ of $F$ outside the set of bad places for $\mathcal{R}_i$ and every $\lambda \in \Lambda_i'$ not lying above the residue characteristic of $v$ we have that
\begin{align*}
\tr \rho_{i, \lambda} (\Frob_v) &= \tr \rho_{i, \lambda_0} (\Frob_v) \\
&= \tr \Ind_{\Gamma_{F_i}}^{\Gamma_F} \sigma_i^{\lambda_0} (\Frob_v) \\
&= \tr \Ind_{\Gamma_{F_i}}^{\Gamma_F} \sigma_{i, \lambda} (\Frob_v),
\end{align*}
where the first equality follows from the independence of $\lambda$ of the characteristic polynomials at the Frobenius elements in the compatible system $\mathcal{R}_i$, and the last equality follows from the usual formula for the trace of an induced representation, see \cite[\S 5.5]{BLGGT14} for instance, and the independence of~$\lambda$~of the characteristic polynomials at the Frobenius elements in the compatible system $\mathcal{S}_i$ . Combining the \v{C}ebotarev density theorem with \cite[\S 12.1, Proposition 3]{bourbaki_algebre_chapitre8} we deduce that~$\rho_{i, \lambda} \cong \Ind_{\Gamma_{F_i}}^{\Gamma_F} \sigma_{i, \lambda}$~for every $\lambda \in \Lambda_i'$. Since $\sigma_{i, \lambda}$ is absolutely Lie-irreducible for each $\lambda$ in the set $\Lambda_i''$ of residual Dirichlet density $1$, this proves that
\[ \mathcal{R} \cong \oplus_{i = 1}^k \Ind_{\Gamma_{F_i}}^{\Gamma_F} \mathcal{S}_i\]
is a Lie-irreducible decomposition for $\mathcal{R}$ over an extension $E'$ of $E$ and the fields $E_i$.
\end{proof}

Combining this with Corollary \ref{independence_Lie-irreducible_decomposition}, we get the following result.

\begin{corollary} \label{independence_automorphic}
Let $\mathcal{R} = \set{\rho_\lambda}_{\lambda \in \Lambda}$ be a pure, regular, totally odd, conjugate self-dual geometric compatible system of rank $n$ representations of $\Gamma_F$ with coefficients in $E$. Assume Conjecture \ref{conjecture_irreducibility_automorphic_compatible_systems}. Then, there exist a finite extension $L$ of $E$, a subset $\Lambda'$ of $\Lambda$ of residual Dirichlet density~$1$, and a connected reductive algebraic subgroup $G$ of $\GL_{n, L}$ defined over $L$ such that for each~$\lambda \in \Lambda'$~there exists a finite place $\mathfrak{l}$ of $L$ lying above $\lambda$ such that $G_{\lambda}^\circ \times_{E_{\lambda}} L_{\mathfrak{l}}$ is conjugate to $G \times_{L} L_{\mathfrak{l}}$ over $L_{\mathfrak{l}}$.
\end{corollary}

\begin{remark}
In Theorem \ref{decomposition_theorem_automorphic}, assuming Conjecture \ref{conjecture_irreducibility_automorphic_compatible_systems} is required in order to apply \cite[Theorem 5.5.1]{BLGGT14} to get potential automorphy of certain $\lambda$-adic representations. This in turn guarantees that such representations can be ``extended'' to a compatible system. Refined potential automorphy results would then eventually imply more general versions of Theorem \ref{decomposition_theorem_automorphic} and Corollary \ref{independence_automorphic}. Notice also that the conclusions of Theorem \ref{decomposition_theorem_automorphic} and Corollary \ref{independence_automorphic} hold unconditionally when the compatible system has extremely regular weights, in the sense of \cite[\S 5.1]{BLGGT14}.
\end{remark}

\begin{remark}
Let $\pi$ be a regular algebraic, conjugate self-dual, cuspidal automorphic representation of $\GL_n (\mathbb{A}_F)$, and let $\mathcal{R}_\pi = \set{\rho_{\pi, \lambda}}_{\lambda \in |E_\pi|}$ be the corresponding compatible system of representations of $\Gamma_F$. Assuming Conjecture \ref{conjecture_irreducibility_automorphic_compatible_systems}, the conclusions of Theorem \ref{decomposition_theorem_automorphic} and Corollary \ref{independence_automorphic} clearly hold for $\mathcal{R}_\pi$. We then have that there exist a finite (possibly trivial) CM extension $F_\pi$ of $F$ and a compatible system $\mathcal{S} = \set{\sigma_{\lambda}}_{\lambda \in \Lambda'}$ of absolutely Lie-irreducible representations of $\Gamma_{F_\pi}$ such that $\mathcal{R}_\pi \cong \mathrm{Ind}_{\Gamma_{F_\pi}}^{\Gamma_F} \mathcal{S}$. Notice in particular that the compatible system $\mathcal{S}$ is by construction the compatible system attached to a regular algebraic, conjugate self-dual, cuspidal automorphic representation $\tau$ of $\GL_m (\mathbb{A}_{F_\pi})$, and so $\pi$ should be thought as the ``automorphic induction'' of $\tau$ from $F_\pi$. In analogy with the case of modular forms of weight $k \geq 2$, we suggest to think of $\pi$ as having ``complex multiplication'' by $F_\pi$. Also, we have that there exists a connected reductive group $G_\pi$ defined over a finite extension of $E_\pi$ which interpolates the groups of connected components of the algebraic monodromy groups of $\rho_{\pi, \lambda}$, for $\lambda$ in a set of places of residual Dirichlet density $1$. In analogy with the Mumford-Tate conjecture, the group $G_\pi$ should be related to a notion of ``Mumford-Tate group'' for $\pi$.
\end{remark}

\section{Compatible systems in the positive characteristic case} \label{independence_compatible_systems_lisse_sheaves}

In this last section we see applications of the abstract independence results of \S \ref{independence_absolutely_irreducible} to positive characteristic settings. We start with the case of compatible systems of representations of the absolute Galois group of a global function field, and we then move to compatible systems of lisse sheaves on a scheme of finite type over a finite field.

Let $p$ be a prime, let $q$ be a power of $p$, and let $C$ be a smooth projective curve, geometrically connected over $\mathbb{F}_q$. Let $F$ be the field of rational functions on $C$, and let $\Gamma_F$ denote the absolute Galois group of $F$. The following preliminary result is essentially a reformulation in the characteristic $p$ setting of \cite[Proposition 3.4.1]{pat12}, which in turn adapts the arguments of \cite[Proposition 1]{katz87}.

\begin{lemma} \label{decomposition_representations}
Let $\rho : \Gamma_F \ra \GL_n (\alg{\Q}_\ell)$ be an irreducible representation. Then either $\rho$ is induced from a representation of an open proper subgroup of $\Gamma_F$, or $\rho$ is Lie-irreducible, or there exist an integer $d \geq 2$ dividing $n$, a Lie-irreducible representation $\sigma$ of $\Gamma_F$ of dimension $n/d$, and an Artin representation $\omega$ of $\Gamma_F$ of dimension $d$ such that $\rho \cong \sigma \otimes \omega$.
\end{lemma}

The proof of this lemma is just the same as the proof of \cite[Proposition 3.4.1]{pat12}, where the key ingredient is the vanishing of $H^2 (\Gamma_F, \Q / \Z)$ given by Tate's theorem. This holds true also in the positive characteristic case, see \cite[Theorem 4]{serre77}.

We prove the following result.

\begin{theorem} \label{decomposition_theorem_curves}
Let $\mathcal{R} = \set{\rho_\lambda}_{\lambda \in \Lambda}$ be a compatible system of semisimple representations of $\Gamma_F$ with coefficients in $E$. Then, $\mathcal{R}$ has a Lie-irreducible decomposition over a finite extension $E'$ of $E$.
\end{theorem}

\begin{proof}
Fix $\lambda_0 \in \Lambda_{\neq p}$. By Lemma \ref{decomposition_representations} there exist finite extensions $F_i$ of $F$, Lie-irreducible representations $\sigma_i : \Gamma_{F_i} \ra \GL_{m_i} (\alg{E}_{\lambda_0})$, Artin representations $\omega_i : \Gamma_{F_i} \ra \GL_{d_i} (\alg{E}_{\lambda_0})$, for $i = 1, \ldots, k$, where $\sum_{i = 1}^k m_i d_i [F_i : F] = n$, such that
\[ \rho_{\lambda_0} \otimes_{E_{\lambda_0}} \alg{E}_{\lambda_0} \cong \oplus_{i = 1}^k \mathrm{Ind}_{\Gamma_{F_i}}^{\Gamma_F} (\sigma_i \otimes \omega_i). \]

For each $i$, choose $\chi_i : \Gamma_{F_i} \ra \alg{E}_{\lambda_0}^\times$ such that the determinant of $\tau_i = \sigma_i \otimes \chi_i$ has finite order. Then, $\tau_i$ extends to a compatible system $\set{\tau_{i, \lambda}}_{\lambda \in \Lambda_{\neq p}}$ by the global Langlands correspondence \cite[Th\'eor\`eme VI.9]{lafforgue02}. Also, by global class field theory we can extend $\chi_i$ to a compatible system $\set{\chi_{i, \lambda}}_{\lambda \in \Lambda_{\neq p}}$. For each $\lambda \in \Lambda_{\neq p}$, let $\sigma_{i , \lambda} = \tau_{i, \lambda} \otimes \chi_{i, \lambda}^{-1}$. We then have that $\mathcal{S}_i = \set{\sigma_{i, \lambda}}_{\lambda \in \Lambda_{\neq p}}$ is a compatible system extending $\sigma_i$.

Let $E'$ be a finite extension of $E$ such that each $\mathcal{S}_i$ has coefficients in $E'$, and each $\omega_i$ can be realised over $E'$. Then, for any place $v$ of $F$ and any $\lambda \in \Lambda_{\neq p}$ not lying above the residue characteristic of $v$ we have
\begin{align*}
\tr \rho_\lambda (\Frob_v) &= \tr \rho_{\lambda_0} (\Frob_v) \\
&= \tr \oplus_{i = 1}^k \Ind_{\Gamma_{F_i}}^{\Gamma_F} (\sigma_i \otimes \omega_i) (\Frob_v) \\
&= \tr \oplus_{i = 1}^k \Ind_{\Gamma_{F_i}}^{\Gamma_F} (\sigma_{i, \lambda} \otimes \omega_i) (\Frob_v),
\end{align*}
by independence of $\lambda$ of the characteristic polynomials at the Frobenius elements, and by the usual formulas for the traces of direct sums, tensor products, and induced representations.
Combining the \v{C}ebotarev density theorem with \cite[\S 12.1, Proposition 3]{bourbaki_algebre_chapitre8}, it then follows that
\[ \mathcal{R} \cong \oplus_{i = 1}^k \Ind_{\Gamma_{F_i}}^{\Gamma_F} (\mathcal{S}_i \otimes \omega_i) \]
is a Lie-irreducible decomposition for $\mathcal{R}$ over $E'$.
\end{proof}

By Corollary \ref{independence_Lie-irreducible_decomposition} we then have the following immediate consequence.

\begin{corollary} \label{independence_positive_characteristic}
Let $\mathcal{R} = \set{\rho_\lambda}_{\lambda \in \Lambda}$ be a compatible system of semisimple representations of $\Gamma_F$ with coefficients in $E$. Then, there exist a finite extension $L$ of $E$, a subset $\Lambda'$ of $\Lambda$ of residual Dirichlet density~$1$, and a connected reductive algebraic subgroup $G$ of $\GL_{n, L}$ defined over $L$ such that for each~$\lambda \in \Lambda'$~there exists a finite place $\mathfrak{l}$ of $L$ lying above $\lambda$ such that $G_{\lambda}^\circ \times_{E_{\lambda}} L_{\mathfrak{l}}$ is conjugate to $G \times_{L} L_{\mathfrak{l}}$ over $L_{\mathfrak{l}}$.
\end{corollary}

Let us now consider the case of compatible systems of lisse sheaves on a scheme of finite type over $\mathbb{F}_q$. Refer to \cite[\S1.1 - \S1.3]{deligne80} for the basic definitions and results in this setting.

Let $X$ be a scheme of finite type, and geometrically connected over $\mathbb{F}_q$. Let $\overline{\eta}$ be a geometric point of $X$, and let $\pi_1^{\mathrm{\acute{e}t}} (X, \overline{\eta})$ be the arithmetic fundamental group of $X$. Let $\ell \neq p$ be a prime. The functor which assigns to each lisse $\alg{\Q}_\ell$-sheaf $\mathcal{L}$ on $X$ its fibre $\mathcal{L}_{\overline{\eta}}$ over $\overline{\eta}$ induces an equivalence between the category of lisse $\alg{\Q}_\ell$-sheaves on $X$ and the category of finite dimensional continuous representations of $\pi_1^{\mathrm{\acute{e}t}} (X, \overline{\eta})$ with coefficients in $\alg{\Q}_\ell$. Via this equivalence, standard notions about representations (e.g semisimplicity, irreducibility, or Lie-irreducibility) can be translated to the context of lisse sheaves.

Let $E$ be a number field, and let $\Lambda$ be a set of finite places of $E$ of residual Dirichlet density $1$. We recall the following definition.

\begin{definition}
A \emph{compatible system} of lisse $E_\lambda$-sheaves on $X$ is a family $\mathcal{R} = \set{\mathcal{L}_\lambda}_{\lambda \in \Lambda}$ of lisse $E_\lambda$-sheaves $\mathcal{L}_\lambda$ on $X$ such that, for any closed point $x \in |X|$ of $X$ and any $\lambda \in \Lambda$ not dividing $p$, the polynomial $\det (1-\Frob_x t, \mathcal{L}_\lambda)$ has coefficients in $E$ and is independent of $\lambda$.
\end{definition}

Any compatible system of lisse $E_\lambda$-sheaves $\mathcal{R} = \set{\mathcal{L}_\lambda}_{\lambda \in \Lambda}$ on $X$ defines a compatible system $\set{\rho_\lambda}_{\lambda \in \Lambda_{\neq p}}$ of representations of $\pi_1^{\mathrm{\acute{e}t}} (X, \overline{\eta})$ with coefficients in $E$, in the sense of Definition \ref{compatible_systems_definition}, where $\rho_\lambda$ is the representation $\mathcal{L}_{\lambda, \overline{\eta}}$, and $\mathcal{X} = |X| \times \Lambda_{\neq p}$. The algebraic monodromy group $G_\lambda$ of $\rho_\lambda$ is called the \emph{arithmetic monodromy group} of $\mathcal{L}_\lambda$. We denote by $\mathfrak{g}_\lambda$ the Lie algebra of the connected component $G_\lambda^\circ$ of the identity of $G_\lambda$. Also, when $X$ is geometrically connected over $\mathbb{F}_q$, for simplicity of notation we omit the base point $\overline{\eta}$ and write just $\pi_1^{\mathrm{\acute{e}t}} (X)$ for the \'etale fundamental group of $X$.

From Corollary \ref{independence_positive_characteristic} we deduce the following $\lambda$-independence result.

\begin{corollary} \label{independence_lisse_sheaves}
Let $X$ be a normal geometrically connected irreducible variety over $\mathbb{F}_q$, and let $\mathcal{R} = \set{\mathcal{L}_\lambda}_{\lambda \in \Lambda}$ be a compatible system of semisimple lisse $E_\lambda$-sheaves on $X$. Then, the following hold.
\begin{enumerate}
\item[$(1)$] There exist a finite extension $L$ of $E$, a subset $\Lambda'$ of $\Lambda$ of residual Dirichlet density~$1$, and a reductive Lie subalgebra $\mathfrak{g}$ of $\gl_{n, L}$ such that for each~$\lambda \in \Lambda'$~there exists a finite place $\mathfrak{l}$ of $L$ lying above $\lambda$ such that $\mathfrak{g}_{\lambda} \otimes_{E_{\lambda}} L_{\mathfrak{l}}$ is conjugate to $\mathfrak{g} \otimes_{L} L_{\mathfrak{l}}$ over $L_{\mathfrak{l}}$.
\item[$(2)$] If $\mathcal{R}$ is tame\footnote{We say that a compatible system $\mathcal{R} = \set{\mathcal{L}_\lambda}_{\lambda \in \Lambda}$ of lisse $E_\lambda$-sheaves on $X$ is \emph{tame} if for any discrete rank $1$ valuation $v$ of the field $\mathbb{F}_q (X)$ of rational functions on $X$ and for any $\lambda \in \Lambda_{\neq p}$ the fixed field inside $\sep{\mathbb{F}_q (X)}$ of the kernel of the representation of $\Gamma_{\mathbb{F}_q (X)}$ defined by $\mathcal{L}_\lambda$ is a tame extension of of $\mathbb{F}_q (X)$ at $v$.}, then there exist a finite extension $L$ of $E$, a subset $\Lambda'$ of $\Lambda$ of residual Dirichlet density~$1$, and a connected reductive algebraic subgroup $G$ of $\GL_{n, L}$ defined over $L$ such that for each~$\lambda \in \Lambda'$~there exists a finite place $\mathfrak{l}$ of $L$ lying above $\lambda$ such that $G_{\lambda}^\circ \times_{E_{\lambda}} L_{\mathfrak{l}}$ is conjugate to $G \times_{L} L_{\mathfrak{l}}$ over $L_{\mathfrak{l}}$.
\end{enumerate}
\end{corollary}

\begin{proof}
Let $\set{\rho_\lambda}_{\lambda \in \Lambda_{\neq p}}$ be the compatible system of representations of $\pi_1^{\mathrm{\acute{e}t}} (X)$ corresponding to $\mathcal{R}$. By \cite[Corollary 3.33]{BGP19} there exist a finite \'etale cover $\varphi : X' \ra X$, and a curve $C$ embedding into $X'$ via some $\imath$ such that $\rho_\lambda (\varphi_* \pi_1^{\mathrm{\acute{e}t}} (X')) = \rho_\lambda ( \imath_* \pi_1^{\mathrm{\acute{e}t}} (C))$, for all $\lambda \in \Lambda_{\neq p}$. Furthermore, if $\mathcal{R}$ is tame, we can take $\varphi = \mathrm{Id}_X$.

For any $\lambda \in \Lambda_{\neq p}$, let $\rho_\lambda' = \varphi_* \circ \rho_\lambda$, and let $G_\lambda'$ be its algebraic monodromy group. Since the morphism $\varphi$ is finite \'etale, then $\varphi_* \pi_1^{\mathrm{\acute{e}t}} (X')$ is an open subgroup of $\pi_1^{\mathrm{\acute{e}t}} (X)$, and so the Lie algebra of $G_\lambda'^{, \circ}$ coincides with the Lie algebra $\mathfrak{g}_\lambda$ of $G_\lambda^\circ$. Then, $\set{\rho'_\lambda}_{\lambda \in \Lambda_{\neq p}}$ is a compatible system of semisimple representations of $\pi_1^{\mathrm{\acute{e}t}} (X')$. Note that, obviously, if $\mathcal{R}$ is tame, and we take $\varphi = \mathrm{Id}_X$, then $G_\lambda' = G_\lambda$.

Up to birational equivalence, we can assume that $C$ is smooth and projective. Let $F$ be the field of rational functions on $C$, and let $\pi_C : \Gamma_F \ra \pi_1^{\mathrm{\acute{e}t}} (X')$ be the composition of the natural surjection of $\Gamma_F$ onto $\pi_1^{\mathrm{\acute{e}t}} (C)$ with $\imath_*$.

For any $\lambda \in \Lambda_{\neq p}$, let $\rho_{C, \lambda} = \rho_\lambda' \circ \pi_C$. Note that $\rho_{C, \lambda} (\Gamma_F) = \rho_\lambda' (\pi_1^{\mathrm{\acute{e}t}} (X'))$, so that the algebraic monodromy groups of $\rho_{C, \lambda}$ and $\rho_\lambda'$ coincide. The family $\set{\rho_{C, \lambda}}_{\lambda \in \Lambda_{\neq p}}$ is a compatible system of semisimple representations of $\Gamma_F$, and then, by Corollary \ref{independence_positive_characteristic} there exist a finite extension $L$ of $E$, a subset $\Lambda'$ of $\Lambda$ of residual Dirichlet density~$1$, and a connected reductive algebraic subgroup $G$ of $\GL_{n, L}$ defined over $L$ such that for each~$\lambda \in \Lambda'$~there exists a finite place $\mathfrak{l}$ of $L$ lying above $\lambda$ such that $G_{\lambda}'^{, \circ} \times_{E_{\lambda}} L_{\mathfrak{l}}$ is conjugate to $G \times_{L} L_{\mathfrak{l}}$ over $L_{\mathfrak{l}}$. If we let $\mathfrak{g}$ be the Lie algebra of $G$, we then get $(1)$. When $\mathcal{R}$ is tame, we take $\varphi = \mathrm{Id}_X$ and we get $(2)$.
\end{proof}

\begin{remark}
A crucial point in the proof of this result is the reduction to the case of curves, which relies on the work of B\"ockle, Gajda, and Petersen in \cite[\S 3.3]{BGP19}. On the other hand, if we assume the variety $X$ to be smooth and projective, the reduction to the case of curves can be achieved (without passing to a finite \'etale cover) by applying iteratively Bertini theorem over finite fields \cite[~Theorem~1.1]{poonen04} and Lefschetz theorem \cite[Exp. XII, Corollaire 3.5]{SGA2}. In this case, we get $\lambda$-independence of the neutral components of the algebraic monodromy groups (and not just $\lambda$-independence of the Lie algebras) even when the compatible system $\mathcal{R}$ is not necessarily tame. A result of this form has been obtained with a different approach by D'Addezio, see \cite[~Theorem~4.3.2]{daddezio17}, without assuming projectivity (but still assuming smoothness). Nevertheless, note that when $X$ is a (not necessarily smooth) curve no projectivity assumption is required, as one can always find a smooth projective curve which is birational to $X$, and apply the above argument to it. In this way, we recover the main independence result of Chin, see \cite[~Theorem~1.4]{chin04}, at a density $1$ set of places, by an alternative method.
\end{remark}

\bibliographystyle{alpha2}
\bibliography{Bibliography.bib}

\end{document}